\documentclass[11pt,a4paper]{amsart}
\usepackage{mathptmx} 
\usepackage[scaled=0.90]{helvet} 
\usepackage{courier} 
\normalfont
\usepackage[T1]{fontenc}
\usepackage{color}
\usepackage{amssymb}
\usepackage{MnSymbol}
\usepackage{graphicx}
\usepackage{epstopdf}
\usepackage{enumitem}
\usepackage{tikz-cd}
\usetikzlibrary{cd}

\usepackage[curve,tips]{xypic}
\SelectTips{eu}{11}
\UseTips

\renewcommand{\epsilon}{\varepsilon}

\newtheorem{theorem}{Theorem}[section]
\newtheorem{proposition}[theorem]{Proposition}
\newtheorem{corollary}[theorem]{Corollary}
\newtheorem{lemma}[theorem]{Lemma}

\newtheorem*{theorema}{Theorem A}
\newtheorem*{theoremb}{Theorem B}

\newtheorem*{theoremb1}{Theorem B1}

\newtheorem*{theoremc}{Theorem C}

\newtheorem*{corollarya}{Corollary A}

\newtheorem*{propositiona}{Proposition}

\theoremstyle{definition}

\newtheorem{definition}[theorem]{Definition}

\theoremstyle{remark}

\newcommand{\iso}{\cong}

\newcommand{\cd}{\operatorname{cd}}
\newcommand{\normal}{\lhd}

\newcommand{\Q}{\mathbb Q}
\newcommand{\Z}{\mathbb Z}

\newcommand{\fpinfty}{{\FP}_{\infty}}

\newcommand{\FP}{\operatorname{FP}}

\renewcommand{\ker}{\operatorname{Ker}}

\renewcommand{\implies}{\Rightarrow}

\renewcommand{\hom}{\operatorname{Hom}}

\newcommand{\End}{\operatorname{End}}

\newcommand{\hd}{\operatorname{hd}}

\newcommand{\aut}{\operatorname{Aut}}

\newcommand{\Maltsev}{{Mal${}'$\!cev}}
\newcommand{\Feldman}{{Fel${}'$\!dman}}

\title[Homological Dimension]{Homological Dimension of Elementary Amenable Groups}
\begin{document}

\author[P. H. Kropholler]{Peter H. Kropholler}
\thanks{P.H.K. was supported in part by  EPSRC grants no EP/K032208/1 and EP/N007328/1.}
\address{Mathematical Sciences, University of Southampton,  UK}
\email{p.h.kropholler@soton.ac.uk}

\author[C. Mart\'inez-P\'erez]{Conchita Mart\'inez-P\'erez}
\address{Conchita Martinez-Perez, University of Zaragoza, Spain}
\thanks{C.M-P. was supported by MINECO grant MTM2015-67781-P and by Gobierno de Aragón and European Regional Development Funds.}
\email{conmar@unizar.es}
\date{\today} 



\begin{abstract}
In this paper we prove that the homological dimension of an elementary amenable group over an arbitrary commutative coefficient ring is either infinite or equal to the Hirsch length of the group. Established theory gives simple group theoretical criteria for finiteness of homological dimension and so we can infer complete information about this invariant for elementary amenable groups. Stammbach proved the special case of solvable groups over coefficient fields of characteristic zero in an important paper dating from 1970.
\end{abstract}


\maketitle

\section{Statement of Result}

We calculate the homological dimension of an elementary amenable group relative to an arbitrary coefficient ring. Throughout the paper, \emph{coefficient ring} means any non-zero commutative ring. We write $\hd_k(G)$ for the homological dimension of the group $G$ over the coefficient ring $k$. When it makes sense, we write $h(G)$ for the Hirsch length of $G$. Hillman established the working definition of Hirsch length for elementary amenable groups in \cite{H0}.

\begin{theorema}
Let $G$ be an elementary amenable group and let $k$ be a coefficient ring. If $\hd_k(G)$ is finite then $\hd_k(G)=h(G)$.
\end{theorema}
 
This answers a question of Bridson and the first author \cite[Conjecture I.1]{BK}. 
Theorem A says that the homological dimension of $G$ is either equal to the Hirsch length or is infinite. Since the below Proposition (which is well known) describes necessary and sufficient conditions for finiteness it follows that we know the homological dimensions of elementary amenable groups.

\begin{propositiona}
Let $G$ be an elementary amenable group and let $k$ be a coefficient ring. Then $\hd_k(G)$ is finite if and only if the following two conditions hold.
\begin{enumerate}
\item
$G$ has no $k$-torsion (meaning that the orders of elements of finite order in $G$ are invertible in $k$).
\item
$h(G)<\infty$.
\end{enumerate}
\end{propositiona}

As observed by \Feldman\ \cite{Feldman} one can draw the following conclusion for cohomological dimension.

\begin{corollarya}
If $G$ is a countable elementary amenable group with no $k$-torsion and with finite Hirsch length then $h(G)\le\cd_k(G)\le h(G)+1$.
\end{corollarya}
\begin{proof} Over any ring, a countably generated flat module has projective dimension at most one. From this, one can deduce that
the inequalities
 $\hd_k(G)\le\cd_k(G)\le\hd_k(G)+1$ hold for any countable group $G$ with $\hd_k(G)<\infty$. Note that the inequality $h(G)\le\cd_\Z(G)$ is proved in \cite[Lemma 1.10 and Theorem 1.11]{Hillman}, building on the analysis of Hillman and Linnell \cite{HL}.
\end{proof}

\subsection*{Subsidiary Results}

Along the way, we have found the need of two subsidiary results which may be of independent interest. The first is a technical splitting theorem which is essentially a refinement of a splitting theorem of Lennox--Robinson and, independently, Zaicev \cite[10.3.2]{LennoxRobinson}. Moreover its proof uses very similar tools developed by Robinson and concerning the vanishing of certain cohomology groups for nilpotent groups (see \cite[10.3.2]{LennoxRobinson}).

\begin{theoremb}
Suppose that $Q$ is a group with subgroups $L$, $M$, $P$ so that the following conditions hold:
\begin{enumerate}
\item
$M$ and $P$ are normal in $Q$.
\item
$LM\subseteq P$.
\item
$M$ is a \Maltsev\ complete nilpotent group of finite Hirsch length.
\item\begin{itemize}
\item either $L=1$ and $P/M$ is nilpotent,
\item or $LM=P$ and $L$ is nilpotent.
\end{itemize}
\end{enumerate}
Then $Q$ has a subgroup $Q_0$ such that 
$L\subseteq	Q_0$,
$P\cap Q_0$ is nilpotent, and $Q_0M=Q$. If moreover $Q/M$ and $L$ are finitely generated, then $Q_0$ can be taken to be finitely generated too.
\end{theoremb}

The second is a consequence of the Bieri--Strebel theory of solvable groups. 
For the statement we need to introduce two items of terminology.
A group is called \emph{locally of type $\fpinfty$} meaning that every finite subset of the group is contained in a subgroup of type $\fpinfty$. As we shall see later, an elementary amenable group which is locally of type $\fpinfty$ and which has Hirsch length $n<\infty$ always has the property that all finitely generated subgroups of Hirsch length $n$ are of type $\fpinfty$ and this is a reason why the property has prominent role. Also, the homological dimension of solvable groups of type $\fpinfty$ is well understood, and this is a second reason why the property is useful to us. Further,
we say that an endomorphism $\alpha$ of an abelian group is \emph{integral} if there are a positive integer $m$ and integers $b_0,\dots,b_{m-1}$ such that
 $$\alpha^m+b_{m-1}\alpha^{m-1}+\dots+b_1\alpha+b_0=0.$$

\begin{theoremc}
Let $G$ be a group with a nilpotent normal subgroup $N$ such that $h(N)<\infty$ and $G/N$ is finitely generated and abelian. Assume further that there is a finite subset $X$ of $G$ such that the following two properties hold.
\begin{enumerate}
\item
$G=\langle N\cup X\rangle$.
\item
Each automorphism of $N_{\textrm{ab}}$ that is induced by conjugation by an element of $X$ is integral.
\end{enumerate}
Then $G$ is locally of type $\fpinfty$.
\end{theoremc}

\subsection*{Organization of the paper}

In the next section we include some background material about solvable groups of finite rank, Hirsch lengths, constructible groups, and inverse duality groups.
Material on nilpotent groups and their \Maltsev\ completions is contained in Section 3. Some reductions for Theorem A are made in Section 4. Of special note, Lemma \ref{LHS} implies that if $G$ is an elementary amenable group with $\hd_k(G)=h(G)<\infty$ then for any subnormal subgroup $H$ of $G$ we also have $\hd_k(H)=h(H)<\infty$. Much of the technical drive in this paper concerns embedding groups as subnormal subgroups of nicer groups so that this Lemma can be used.

Theorem B is proved in Section 5 using classical cohomological vanishing. 
Theorem C is explained in Section 6, and Theorem A is proved in a special case in Section 7. The proof of Theorem A is completed using some further applications of Theorem B in Section 8. In broad outline the idea is to embed the original group, by using Theorem B, into a group satisfying the hypotheses of Theorem C. Then we can use established theory of solvable groups of type $\fpinfty$ to prove Theorem A.

\subsection{Acknowledgement}
We are indebted to an anonymous referee for noticing a gap in our original proof of Theorem A and for showing us in detail that the argument could be corrected by using the notion of \emph{integral endomorphism}. Thanks to the referee, this concept now plays a central role in the proof of the Theorem. In addition, we are grateful to him/her for many improvements to the exposition throughout the paper.
{\iffalse a number of important improvements to the original draft of this paper. In particular, the referee proposed the formulation of Theorem C that we now include. The use of the concept or integral endomorphisms of abelian groups makes this theorem both more elegant in statement and more easily deduced from the Bieri--Strebel theory. The referee has also pointed out a number of places where are arguments can be simplified or improved. And finally we thank the referee for pointing out a gap in our original arguments at the end of section 7 as well as providing detailed proposals for how to fill this gap. All these points have been taken into account in the present work. }\fi

\section{Background Material and Historical Remarks}

Recall that the class of \emph{elementary amenable groups} is the smallest class of groups containing all finite and all abelian groups, that is also closed under group extensions and directed unions.

\subsection*{Hirsch Length and Homological Dimension} 
The connection between Hirsch length and homological dimension of solvable groups was established by Stammbach who proved that $\hd_K(G)=h(G)$ whenever $G$ is solvable and $K$ is a field of characteristic zero:
his elegant calculation uses exterior powers of abelian groups \cite{Stammbach}. This work, published in 1970, was quickly followed by important work of \Feldman, and amongst other things \Feldman\ makes a claim that Stammbach's characteristic zero calculation can be extended to positive characteristic. However, this claim cannot be substantiated in the way that \Feldman\ proposes. Bieri \cite{BieriQMC} gives a detailed account of Stammbach's result but makes no comment how a calculation of homological dimension of solvable groups in positive characteristic might proceed. In fact, modules witnessing homological dimension in positive characteristic are significantly more complicated than those used by Stammbach. Complication of some kind is unavoidable in the light of \cite[Lemma I.6]{BK}.

\subsection*{The Hillman--Linnell Theorem}

In the solvable case, the following result combines work of \Maltsev, Gruenberg and Robinson. We refer the reader to \cite[\S5.2]{LennoxRobinson} for commentary and proof. 
Hillman and Linnell  \cite{HL} extended this to the larger class of elementary amenable groups, and for this general case we refer the reader to \cite[Theorem 1.9]{Hillman}. Wehrfritz has given an alternative short and explicit account of this result in \cite{BAFW}.

\begin{lemma}\label{hl}
Let $G$ be an elementary amenable group. If $h(G)$ is finite then $G$ has characteristic subgroups $T\subset N\subset H\subset G$ such that the following hold. 
\begin{enumerate}
\item
$T$ is the unique largest normal locally finite subgroup of $G$
\item
$N/T$ is the Fitting subgroup of $G/T$ and it is torsion-free and nilpotent
\item
$H/N$ is a finite rank free abelian group
\item
$G/H$ is finite.
\end{enumerate}
The above conditions uniquely determine the subgroups $T$ and $N$. However, if $G/N$ has a non-trivial finite normal subgroup, there is not necessarily a natural condition specifying $H$ uniquely. \qed
\end{lemma}

Another useful way of understanding Hirsch length is as follows.

\begin{lemma}\label{2.2}
Let $G$ be an elementary amenable group. Then $G$ has finite Hirsch length if and only if there is a series $1=G_0\normal G_1\normal G_2\normal\dots\normal G_n=G$ in which the factors are either cyclic or locally finite. Moreover, when these conditions hold then $G$ has finitely generated subgroups with the same Hirsch length.
\end{lemma}
\begin{proof}
We comment only on the last point. Suppose that 
$1=G_0\normal G_1\normal G_2\normal\dots\normal G_n=G$ is a series with cyclic or locally finite factors. Let $J=\{j;\ G_j/G_{j-1}\cong\Z\}$ and for $j\in J$ choose $g_j$ to be a generator of $G_j$ modulo $G_{j-1}$. Then the subgroup $\langle g_j;\ j\in J\rangle$ is finitely generated of the same Hirsch length as $G$.
\end{proof}

\subsection*{Cohomological Dimension and Constructible Groups}
Calculations of cohomological dimension for solvable groups are harder. The theory is well developed in characteristic zero and it is known that the elementary amenable groups which satisfy $\cd_\Z(G)=h(G)<\infty$ are precisely the torsion-free virtually solvable groups that are constructible (constructable) in the sense of Baumslag and Bieri \cite{BB}. A version of this fact was conjectured by Gildenhuys and Strebel \cite{GS} and proved by the first author \cite{K}. Subsequently this led to a proof that elementary amenable groups of type $\fpinfty$ over $\Z$ are constructible \cite{K2} and to the construction of classifying spaces for proper actions for such groups, see \cite{K,KMPN,KM}. (Results in \cite{K2,KM} also apply considerably beyond the elementary amenable case.)

\subsection*{Inverse Duality Groups}
Key results of \Feldman\ \cite{Feldman} are covered in Bieri's notes and are used by Brown and Geoghegan \cite{BG} to establish the following fundamental result. 
 
\begin{theorem}[The Inverse Duality Theorem]
Let $G$ be a constructible elementary amenable group (that is a group with a subgroup of finite index that can be built up from the trivial group with a finite number of ascending HHN-extensions). If $G$ is torsion-free then $G$ is an inverse duality group.  For such a group, it holds that $\hd_k(G)=\cd_k(G)=h(G)<\infty$ for all coefficient rings $k$.
\end{theorem}

We refer the reader to Bieri's notes for an explanation of cohomological duality and in particular the notion of \emph{inverse duality group}. The inverse duality theorem holds for a wider class of groups that can be described in terms of fundamental groups of graphs of groups.

\subsection*{Baer's Class of Polyminimax Groups} A virtually solvable group $G$ is called \emph{polyminimax} if it has a series $1=G_0\normal G_1\normal\dots\normal G_n=G$ in which the factors are cyclic, quasicyclic or finite. Following Baer's original work \cite{Baer} the term polyminimax has usually been abbreviated to \emph{minimax}. Note that every virtually solvable constructible group is polyminimax \cite{BB}.

\subsection*{Locally $\mathfrak X$-Groups}
If $\mathfrak X$ is a class of groups or a group-theoretical property then by a \emph{locally $\mathfrak X$-group} we mean a group all of whose finite subsets are contained in $\mathfrak X$-subgroups. 
When $\mathfrak X$ is a subgroup closed class or property then locally $\mathfrak X$-groups are exactly those groups whose finitely generated subgroups belong to $\mathfrak X$.
For example, locally finite groups are groups all of whose finitely generated subgroups are finite. Locally $\fpinfty$ groups are groups all of whose finitely generated subgroups are contained in subgroups that have type $\fpinfty$. But note that a locally $\fpinfty$ group can have finitely generated subgroups that are not of type $\fpinfty$.

\subsection*{Commensurate Subgroups}

Two subgroups $H$ and $K$ of a group $G$ are said to be \emph{commensurate} if $H\cap K$ has finite index in both $H$ and $K$. 

\begin{lemma}\label{comm}
Let $G$ be an elementary amenable group of Hirsch length $n<\infty$ that is locally polyminimax. 
Then $G$ has finitely generated subgroups of the Hirsch length $n$ and all such subgroups are commensurate with each other.
\end{lemma}
\begin{proof}
Suppose first that $G$ is a finitely generated polyminimax group and that $H$ is a subgroup of Hirsch length $n$. A result of Robinson which uses the fact that $G$ is virtually solvable with finite abelian ranks (see \cite[Theorem 3.1]{DFOB})
shows that $H$ has finite index in $G$ in this case.

In general, using Lemma \ref{2.2} we can choose a finitely generated subgroup $H$ of $G$ with $h(H)=h(G)$. If $F$ is any finite subset of $G$ then 
$$h(G)=h(H)\le h(\langle H\cup F\rangle)\le h(G)$$ so $h(H)=h(\langle H\cup F\rangle)$ and applying the above argument to the subgroup $H$ of the finitely generated group $\langle H\cup F\rangle$ we deduce that $H$ has finite index in $\langle H\cup F\rangle$. Now if $K$ is any other finitely generated subgroup of $G$ of the same Hirsch length then we can apply this argument to $\langle H\cup K\rangle$ to deduce that $H$ and also $K$ both have finite index in $\langle H\cup K\rangle$. In particular it follows that $H$ and $K$ are commensurate.
\end{proof}

Putting Lemma \ref{comm}  together with the elementary observation that the $\fpinfty$ property of a subgroup is inherited by any commensurate subgroup, we deduce the following.

\begin{corollary}\label{cor:FPinfty}
Let $G$ be an elementary amenable group of Hirsch length $n<\infty$ that is locally of type $\fpinfty$. Then all the finitely generated subgroups of Hirsch length $n$ are of type $\fpinfty$ and are commensurate with each other. \qed
\end{corollary}

\section{Nilpotent Groups} We write $\gamma_i(G)$ for the $i$th term of the lower central series of $G$. It is defined inductively by $\gamma_1(G)=G$ and then by taking commutators: $\gamma_{i+1}(G)=[\gamma_i(G),G]$. We use a \emph{right-handed} convention for commutators, namely; for group elements $x$ and $y$, we write $[x,y]:=x^{-1}y^{-1}xy$. 

The \emph{free nilpotent group on $d$ generators of class $c$} is defined to be $F/\gamma_{c+1}(F)$ where $F$ is the free group on $d$ generators. So long as $d\ge2$, this group does indeed have class $c$. Of course if $d\le1$ then it is cyclic. 
Note that there is no particular reason for $d$ to be finite: the theory makes sense for any cardinal number $d$. The class $c$ is always a non-negative integer. 

Every nilpotent group $G$ has a \Maltsev\ completion denoted by $G^\Q$ which is uniquely determined, \cite[2.1.1]{LennoxRobinson}. The theory is further developed in \cite[\S2.1]{LennoxRobinson}. There is a natural map from $G$ to its \Maltsev\ completion and this is injective precisely when $G$ is torsion-free. The free nilpotent groups are torsion-free and so embed into their \Maltsev\ completions. The \Maltsev\ completion is functorial in the sense that a homomorphism $G\to H$ from one nilpotent group to another induces a uniquely determined homomorphism $G^\Q\to H^\Q$ between their \Maltsev\ completions. For further information we refer the reader to Hilton's paper \cite{Hilton73}. In particular Hilton provides a proof that there are natural isomorphisms $\left(\gamma_i(N)/\gamma_{i+1}(N)\right)^\Q\iso\gamma_i(N^\Q)/\gamma_{i+1}(N^\Q)$ for $N$ nilpotent and $i\ge1$.

\begin{definition}
By a \emph{free \Maltsev\ complete group} on $d$ generators of class $c$ we shall mean the \Maltsev\ completion of the free nilpotent group on $d$ generators of class $c$. Whenever we use the term \Maltsev\ complete it is to be understood that the group in question is nilpotent.
\end{definition}

The following is well known.

\begin{lemma}\label{lem:factor}
Let $G$ be a free \Maltsev\ complete group of class $c$. Let $H\to K$ be any surjective homomorphism where $H$ is \Maltsev\ complete of class at most $c$. Then every homomorphism from $G$ to $K$ factors through $H$.
\end{lemma}
\begin{proof} Let $d$ be the dimension of the rational vector space $G/\gamma_2(G)$.
Let $F$ be a free nilpotent group of class $c$ on $d$ generators. Now, we can identify $G$ with $F^\Q$. The universal property of $F$ ensures that the composite map $F\to F^\Q=G\to K$ factors through $H$. On passing to \Maltsev\ completion we have an induced map $F^\Q=G\to H^\Q=H$ as required.
\end{proof}

The following Lemma, also well known, is required for the proof of an important result about \Maltsev\ complete nilpotent groups, namely Proposition \ref{Maltsevcover} below.

\begin{lemma}
Let $G$ be a nilpotent group and let $H$ be a subgroup such that $H\gamma_2(G)=G$. Then $H=G$.
\end{lemma}
\begin{proof}
Since $G$ is nilpotent, $H$ is subnormal and so we may replace $H$ by its normal closure and without loss of generality we may assume that $H$ is normal. The quotient $G/H$ is therefore both perfect and nilpotent which implies $H=G$.
\end{proof}

\begin{definition}\label{pro-h}
For any group $G$ and any positive integer $m$ we say that an autormophism $\phi$ of $G$ is \emph{$m$-powering} provided that for all $g\in G$,
$$g^\phi\in g^m\gamma_2(G).$$ 
By a \emph{powering automorphism} we mean an automorphism that is $m$-powering for some $m\ge1$. When $G$ is a group such that $G/\gamma_2(G)$ is not a torsion group then a powering automorphism is $m$-powering for a uniquely determined rational number $m$. Note that this uniqueness property therefore holds when $G$ is a non-trivial \Maltsev\ complete group.
\end{definition}

Let $N$ be a group and denote by $U$ the group of 1-powering automorphisms of $N$. For each $j$ and each automorphism $\phi$ of $N$, let ${}^j\phi$ denote the induced automorphism of $N/\gamma_{j+1}(N)$. The assignment $\phi\mapsto{}^j\phi$ determines a homomorphism $$U\to\aut(N/\gamma_{j+1}(N))$$ and we write $U_j$ for the kernel of this homomorphism. 
In this way we obtain a descending chain
$$U=U_1\ge U_2\ge U_3\dots$$
of normal subgroups of $U$. With this notation we have:

\begin{proposition}\label{prop:ref1}
Let $N$ be a group and let $\Gamma=\aut(N)$. Let $U$ be the group of $1$-powering automorphisms of $N$. For each $j\ge1$ let $U_j$ denote the subgroup of $\Gamma$ comprising all automorphisms that induce the identity automorphism of $N/\gamma_{j+1}(N)$. The series $$U=U_1\ge U_2\ge\dots$$ is a descending central series of normal subgroups of $\Gamma$. Moreover, using the left $\Gamma$-module structure on 
$U_j/U_{j+1}$ arising from conjugation, $U_j/U_{j+1}$ is isomorphic to a submodule of
$$\hom(N_{\textrm{ab}},\gamma_{j+1}(N)/\gamma_{j+2}(N))$$ regarded as a left $\Gamma$-module using the diagonal action. It follows that if $N$ is nilpotent of class $c\ge1$ then $U$ is nilpotent of class $c-1$.
\end{proposition}
\begin{proof}
Let $\alpha\in U_j$ and denote by $\alpha^*$ the automorphism of $N/\gamma_{j+2}(N)$ induced by $\alpha$. The epimorphism from the $(j+1)$st tensor power of $N_{\textrm{ab}}$ to $\gamma_{j+1}(N)/\gamma_{j+2}(N)$ induced by the iterated commutator map reveals that $\alpha$ induces the identity map on $\gamma_{j+1}(N)/\gamma_{j+2}(N)$. That is to say that the diagram
\[
\xymatrix{
1\ar[r]&
\gamma_{j+1}(N)/\gamma_{j+2}(N)\ar[r]\ar@{=}[d]&
N/\gamma_{j+2}(N)\ar[r]\ar[d]^{\alpha^*}&
N/\gamma_{j+1}(N)\ar[r]\ar@{=}[d]&
1\\
1\ar[r]&
\gamma_{j+1}(N)/\gamma_{j+2}(N)\ar[r]&
N/\gamma_{j+2}(N)\ar[r]&
N/\gamma_{j+1}(N)\ar[r]&
1
}
\]
commutes. Now define the homomorphism 
$\phi_\alpha:N_{\textrm{ab}}\to\gamma_{j+1}(N)/\gamma_{j+2}(N)$
by $$\phi_\alpha(g[N,N])=\alpha(g)g^{-1}\gamma_{j+2}(N).$$
 It can then be verified that the map
$\alpha\mapsto\phi_\alpha$ defines a homomorphism 
$$U_j\to\hom(N_{\textrm{ab}},\gamma_{j+1}(N)/\gamma_{j+2}(N))$$
of $\Gamma$-operator groups whose kernel is $U_{j+1}$. Moreover, since $U$ acts trivially on $\hom(N_{\textrm{ab}},\gamma_{j+1}(N)/\gamma_{j+2}(N))$, the series
$U=U_1\ge U_2\ge\dots$ is central. Hence, if $N$ is nilpotent of class $c\ge1$ then $U$ is also nilpotent and has class $\le c-1$. Since $U$ contains a copy of $N/Z(N)$ we conclude that $U$ has class exactly $c-1$.
\end{proof}

For a group $G$ and a subgroup $H$, let $\aut(G;H)$ denote the set of those automorphisms of $G$ which restrict to automorphisms of $H$. This is a subgroup of $\aut(G)$ and the restriction map affords a homomorphism $\aut(G;H)\to\aut(H)$.
If $H$ is normal in $G$ then elements of $\aut(G;H)$ naturally induce automorphisms of $G/H$ and there is a homomorphism $\aut(G;H)\to\aut(G/H)$.

\begin{proposition}\label{Maltsevcover}
Let $N$ be a \Maltsev\ complete nilpotent group with $h(N)<\infty$. Then there exist a free \Maltsev\ complete nilpotent group $\widehat N$ of the same class as $N$ and an epimorphism $\pi:\widehat N\to N$ such that the following two statements hold.
\begin{enumerate}
\item The induced map $\pi_{\textrm{ab}}:\widehat N_{\textrm{ab}}\to N_{\textrm{ab}}$ is an isomorphism.
\item
The induced homomorphism $\bar\pi:\aut(\widehat N;\ker\pi)\to \aut(N)$ is surjective, and the inverse image under $\bar\pi$ of the subgroup of $1$-powering automorphisms of $N$ is the subgroup of all $1$-powering automorphisms in $\aut(\widehat N;\ker\pi)$.
\end{enumerate}
\end{proposition}
\begin{proof}
Let $d=\dim_\Q N_{\textrm{ab}}$ and let $c$ be the class of $N$. Take $F$ to be the free nilpotent group of class $c$ on a generating set $\mathcal S$ of cardinality $d$. Also, let $\mathcal T$ be a subset of $N$ with cardinality $d$ whose image in $N_{\textrm{ab}}$ is a $\Q$-basis of $N_{\textrm{ab}}$. Set $\widehat N:= F^\Q$. Fix a bijection from $\mathcal S$ to $\mathcal T$: this gives rise to a unique homomorphism $F\to N$. The universal property of \Maltsev\ completions now yields a homomorphism $\pi:\widehat N\to N$ satisfying property (i) and this is surjective by Lemma 3.3.

For $\alpha\in\aut(N)$, the projective property (Lemma 3.2) supplies a homomorphism $\widehat\alpha:\widehat N\to\widehat N$ so that the diagram
\[
\xymatrix{
\widehat N\ar[r]^\pi\ar[d]^{\widehat\alpha} & N\ar[d]^\alpha\\
\widehat N\ar[r]^\pi& N.
}
\]
commutes.
By considering the abelianization of this commutative diagram, we see that $\widehat\alpha_{\textrm{ab}}$ is an automorphism. Thus, by Lemma 3.3, $\widehat\alpha$ is surjective and so also injective. This implies that $\widehat\alpha\in\aut(\widehat N;\ker\pi)$ and $\bar\pi(\widehat\alpha)=\alpha$. Thus $\bar\pi$ is surjective. In addition, the abelianization of the diagram reveals that $\widehat\alpha_{\textrm{ab}}$ is the identity map on $\widehat N_{\textrm{ab}}$ if and only if $\alpha_{\textrm{ab}}$ is the identity map on $N_{\textrm{ab}}$. Statement (ii) is now proved.
\end{proof}

 The following result will be useful for our arguments.

\begin{lemma}\label{new3.12}
Let $N$ be a nilpotent group and $U$ the group of $1$-powering automorphisms of $N$. Let $\Gamma$ be a subgroup of $\aut(N)$ that acts nilpotently on $N$.
\begin{enumerate}
\item
If $N$ is torsion-free then so is $\Gamma$.
\item
If $N$ is \Maltsev\ complete then so is $U$.
\end{enumerate}
\end{lemma}
\begin{proof} We give a brief outline. For (i) we can use the upper central series of $N$ to reduce to the case when $N$ is torsion-free and abelian in which case the automorphisms that act nilpotently behave like upper unitriangular matrices and clearly have infinite order whenever non-trivial. For (ii) one can also use induction on the length of the upper central series of $N$.
\end{proof}

The next Lemma illustrates that \Maltsev\ complete groups have an abundance of powering automorphisms, and is an easy consequence of Lemma \ref{lem:factor}. It plays a vital role in our arguments.

\begin{lemma}
Let $m$ be a positive integer and let $F$ be a free \Maltsev\ complete group. Then there exists an $m$-powering automorphism of $F$.
\end{lemma}

\section{Reductions for the Proof of Theorem A}

Locally $\fpinfty$ groups have an important role to play in our arguments. 
Before returning to this class we need to mention the key inequalities of a more elementary nature.

\begin{lemma}[\cite{BK}, Theorem I.2]\label{bk}
Let $G$ be an elementary amenable group. Then the following are equivalent.
\begin{enumerate}
\item
$\hd_k(G)<\infty$.
\item
$h(G)<\infty$ and $G$ has no $k$-torsion.
\end{enumerate}
When these conditions hold, we have $\frac{h(G)}{2}\le\hd_k(G)\le h(G)$. \qed
\end{lemma}

We refer the reader to \cite{BieriQMC} for the standard theory summarized in the following proposition.

\begin{proposition}\label{tipsy}
The following hold for any group $G$ and any coefficient ring $k$ such that $\hd_k(G)<\infty$.
\begin{enumerate}
\item
If $H$ is a subgroup of $G$ then $\hd_k(G)\ge\hd_k(H)$.
\item
If $H$ has finite index in $G$ then $\hd_k(G)=\hd_k(H)$.
\item
If $T$ is a normal locally finite subgroup of $G$ then
$\hd_k(G/T)=\hd_k(G)$.
\item
If $H$ is a subgroup with the property that $|\langle H\cup F\rangle:H|<\infty$ for all finite subsets of $G$ then $\hd_k(G)=\hd_k(H)$. \qed
\end{enumerate}
\end{proposition}

We now return to groups that are locally of type $\fpinfty$.
The following result is fundamental to our approach to Theorem A.

\begin{proposition}\label{fpinftys}
Let $G$ be an elementary amenable group that is locally of type $\fpinfty$. If $\hd_k(G)<\infty$ then $\hd_k(G)=h(G)$.
\end{proposition}
\begin{proof} Assume that $\hd_k(G)$ is finite. Then $G$ has no $k$-torsion and $h(G)$ is finite.
Let $H$ be a finitely generated subgroup of $G$ of the same Hirsch length. Then $H$ is of type $\fpinfty$ by Corollary \ref{cor:FPinfty}. Therefore $H$ is constructible and is virtually an inverse duality group over $\Z$. It follows that $h(G)=h(H)=\hd_k(H)\le\hd_k(G)$. Lemma \ref{bk} gives the reverse inequality $\hd_k(G)\le h(G)$ and we deduce that $\hd_k(G)=h(G)$. 
\end{proof}

Finally in this section, we record a simple spectral sequence argument which can be used to show that the validity of the conclusion of Theorem A is inherited by subnormal subgroups.

\begin{lemma}\label{LHS}
Let $G$ be an elementary amenable group with $\hd_k(G)=h(G)<\infty$ and let $N$ be a normal subgroup of $G$ such that $\hd_k(G/N)<\infty$. Then $\hd_k(G/N)=h(G/N)$ and $\hd_k(N)=h(N)$. Furthermore, if $N$ is any normal subgroup of $G$, the conclusion $\hd_k(N)=h(N)$ holds whether or not $\hd_k(G/N)<\infty$.
\end{lemma}
\begin{proof}
By the Lyndon--Hochshild--Serre spectral sequence we have
$$\hd_k(G)\leq\hd_k(G/N)+\hd_k(N).$$ 
By Proposition \ref{tipsy}(i) we know $\hd_k(N)<\infty$ and so Lemma \ref{bk} gives
$$\hd_k(G/N)\leq h(G/N)\textrm{ and }\hd_k(N)\le h(N).$$
Putting the pieces together
we have
$$h(G)=\hd_k(G)\leq\hd_k(G/N)+\hd_k(N)\leq h(G/N)+h(N)=h(G)$$
which gives the desired result in case $\hd_k(G/N)<\infty$.
In general, $G/N$ may not have finite homological dimension, but it is an elementary amenable group of finite Hirsch length and by Lemma \ref{hl} there exist subgroups $N_1\normal G_1$ such that $N\le N_1$, $G_1$ has finite index in $G$, $N_1/N$ is locally finite, and $G_1/N_1$ is torsion-free. By Proposition \ref{tipsy}(iv), $\hd_k(N)=\hd_k(N_1)$ and by \ref{tipsy}(ii), $\hd_k(G)=\hd_k(G_1)$. Note also that $N$ and $N_1$ have the same Hirsch length.

Now we can apply the above argument to the situation that $N_1$ is a normal subgroup of $G_1$ because $G_1/N_1$ being torsion-free of finite Hirsch length does indeed have finite homological dimension. This shows that $\hd_k(N_1)=h(N_1)$, (and that $\hd_k(G_1/N_1)=h(G_1/N_1)$ but we do not need this). The final statement of our lemma now follows by combining with the equalities $\hd_k(N)=\hd_k(N_1)$ and $h(N)=h(N_1)$.
\end{proof}

As remarked at the end of Section 1 this has the following consequence.

\begin{corollary}\label{subnormals}
Let $G$ be an elementary amenable group with $\hd_k(G)=h(G)<\infty$. If $H$ is a subnormal subgroup of $G$ then $\hd_k(H)=h(H)<\infty$.
\end{corollary}

Of course the corollary is also a consequence of Theorem A, but its role in the proof of Theorem A is significant. In the light of this, following some basic reductions the proof of Theorem A is mainly concerned with embedding a group as a subnormal subgroup of a locally $\fpinfty$ group.

\section{Cohomological Vanishing Results, Splittings, and Theorem B} 

We shall need the following key vanishing result of Robinson in the special case when $S=\Q$ and $\dim_\Q M<\infty$.

\begin{theorem}[\cite{LennoxRobinson} 10.3.1 and 10.3.2]\label{vanishing} Let $S$ be a commutative ring, $G$ a nilpotent group, and $M$ an $SG$-module. If either
\begin{itemize}
\item $M$ is noetherian and $M_G=0$, or
\item $M$ is artinian and $M^G=0$, 
\end{itemize}
then $H^n(G,M)=H_n(G,M)=0$ for every $n$. \qed
\end{theorem}

\subsection*{The Proof of Theorem B} We are now in a position to prove the technical splitting theorem that is required for our solution to the homological dimension calculation.

\subsection*{Notational Remark} In the following proof we write $A^B$ to indicate the set of $B$-fixed points in $A$ in case $B$ is a group acting on $A$. 
We use the right-handed convention for conjugation: $x^y=y^{-1}xy$ and note that $[x,y]=x^{-1}x^y$. 

\begin{proof}[Proof of Theorem B]
We proceed by induction on $h:=h(M)$. 
If $h=0$ then $M$ is trivial and we set $Q_0:=Q$: note that \textit{(iv)} implies $P$ is nilpotent in this case and the result is immediate. Suppose now that $h>0$.
Choose $K$ to be a non-trivial divisible subgroup of $M$ of least possible Hirsch length subject to being normalized by $Q$. Note that $K$ must be abelian and therefore can be viewed as a $\Q Q$-module. As such $K$ is irreducible and the action of $Q$ descends to and action of $Q/K$.
 By induction, there is a subgroup $Q_1$ of $Q$ such that 
\begin{itemize}
\item $LK\subseteq P\cap Q_1$,
\item $(P\cap Q_1)/K$ is nilpotent, and
\item  $Q_1M=Q$.
\end{itemize}
 At an important later step in this proof we shall want to restrict the action of $Q$ to $M$ and view $K$ as a $\Q M$-module. The choice of $K$ ensures that $K$ lies in the centre of $M$. We write $P_1:=P\cap Q_1$. As $MP_1=P\cap Q_1M=P\cap Q=P$ and $K^M=K$ we deduce
$$K^P=K^{P_1}.$$

The fact that $P$ is normal in $Q$ implies that $K^P$ is a $\Q Q$-submodule of $K$ thus the irreducibility of $K$ implies that either $K^P=K^{P_1}=0$ or $K^P=K^{P_1}=K$. In the later case, as $P_1/K$ is nilpotent we deduce that also $P_1$ is and we only have to take $Q_0:=Q_1$. Then $Q_0\cap P=Q_1\cap P=P_1$ is nilpotent, $L\leq Q_0$ and $Q_0M=Q_1M=Q$. If we have the extra hypothesis on the finite generation of $Q/M$ and $L$ then, in the inductive step we also deduce that $Q_1/K$ is finitely generated. Choose a set of lifts $X$ to $Q_1$ of a finite generating system of $Q_1/K$ and a finite generating system $Y$ of $L$ and let now $Q_0$ be the subgroup of $Q_1$ generated by $X\cup Y$. 
Then $Q_0$ is finitely generated and $L\leq Q_0$. Also, $Q_0\leq Q_1$ thus $Q_0\cap P\leq Q_1\cap P=P_1$ is nilpotent. Finally, by construction $Q_0K=Q_1$ thus
$$Q_0M=Q_0KM=Q_1M=Q.$$

So we may assume now that $K^P=K^{P_1}=0$. Then Theorem \ref{vanishing} shows that $$H^*(P_1/K,K)=0\eqno(*)$$ Using the spectral sequence $H^*(Q_1/P_1,H^*(P_1/K,K))\implies H^*(Q_1/K,K)$  with $(*)$, we find that $H^*(Q_1/K,K)$ also vanishes. In particular $H^2(Q_1/K,K)=0$ and there is a splitting: there exists a group $Q_0$ of $Q_1$ such that $Q_0K=Q_1$ and $Q_0\cap K=\{1\}$. Note that $P_1\cap Q_0$ is a complement to $K$ in $P_1$ and is therefore nilpotent. Moreover, $Q_0M=Q_1KM=Q$.
In the case when $L=1$ there is nothing else to prove.

So we assume now $LM=P$. At this point, we only need to establish the conclusion $L\subseteq Q_0$, and while this may not be true of the $Q_0$ we are currently entertaining, we show next that we can replace $Q_0$ with a conjugate subgroup to achieve our goal. We begin by observing that, since $M$ centralizes $K$, we have $K^L=K^P=1$. Because $L$ is nilpotent, this means $L\cap K=1$; that is, $L$ is a complement to $K$ in $P_1$. But all such complements must be conjugate because $H^1(P_1/K,K)=0$. Thus there is a $k\in K$ such that $(P_1\cap Q_0)^k=L$. Hence we can replace $Q_0$ with $Q_0^k$ to ensure that $L\subseteq Q_0$.

To finish, note that if we have the extra hypothesis on finite generation then $Q_1/K=Q_0K/K\cong Q_0$ is finitely generated.
\end{proof}

\section{Integral Automorphisms and Theorem C}

Integral automorphisms play a role in Theorem C and also in the next section where we prove a special case of Theorem A. We therefore establish some general theory here.

\begin{definition}\label{def:ref1} Let $A$ be an abelian group. An endomorphism $\alpha$ of $A$ is \emph{algebraic} (resp. \emph{integral}) if there is a polynomial (resp. monic polynomial) $f\in\Z[x]$ such that $f(\alpha)=0$. Let $N$ be a nilpotent group. We say that an endomorphism of $N$ is \emph{integral} if it induces an integral endomorphism on $N_{\mathrm{ab}}$.\end{definition}

The following two lemmas are elementary and the proofs are omitted.

\begin{lemma}\label{lem:ref1}
Let $\alpha$ be an endomorphism of an abelian group $A$ and $B$ a subgroup of $A$ such that $\alpha(B)\subseteq B$. Then $\alpha$ is integral if and only if it induces integral endomorphisms of both $B$ and $A/B$. 
\end{lemma}

\begin{lemma}\label{lem:ref3} Let $A$ be an abelian group and $\alpha\in\End(A)$. If $\alpha$ is algebraic, then there is an integer $m_1\geq 1$ such that $m\alpha$ is integral for every multiple $m$ of $m_1$.
\end{lemma}

\begin{lemma}\label{lem:nuevo4} Let $N$ be a nilpotent group and let $\alpha$ be an integral endomorphism of $N$. If $H$ is a subgroup of $N$ such that $\alpha(H)\subseteq H$, then the restriction of $\alpha$ to $H$ is integral.
\end{lemma}
\begin{proof} This follows by the same reasoning used to establish \cite[Lemma 3.24 (ii)]{KL}.
\end{proof}

\begin{proof}[The proof of Theorem C]
The key results from the Bieri--Stebel that we will invoke are \cite[Theorem 4.6]{bieristrebel1981crelle} and \cite[Theorem A(iii)]{BS}. The notation and terminology we employ are also from those two articles. Let $Y$ be an arbitrary finite subset of $G$ and set $H=\langle X\cup Y\rangle$. We will show that $H$ is of type $\fpinfty$. Put $Q=G/N$ and let $\epsilon:G\to Q$ be the quotient map. Take $X_0$ to be  a subset of $X$ that is mapped by $\epsilon$ to a maximal linearly independent subset of $Q$. Let $Q_0$ be the subgroup of $Q$ generated by $\epsilon(X_0)$. Note that $Q_0$ is then a free abelian group of finite index in $Q$ and $\epsilon$ restricts to an epimorphism $H\to Q$ with kernel $H\cap N$. Since the elements of $X_0$ induce integral automorphisms on $N$ they also induce integral automorphisms on 
$H\cap N$ by Lemma \ref{lem:nuevo4}. Let $A:=(H\cap N)_{\mathrm{ab}}$. According to \cite[Theorem 4.6]{bieristrebel1981crelle}, the Bieri--Strebel invariant $\Sigma^c_A$ is contained in the open hemisphere $H_q=\{[v];\ v(q)>0\}$ for $q$ equal to the product of the elements of $\epsilon(X_0)$. Therefore by \cite[Theorem A(iii)]{BS}, $H$ is of type $\fpinfty$.
\end{proof}

\section{A special case of Theorem A}

Our goal in this section is to prove Theorem A in the special case when $G$ has the following structure:
\begin{itemize}
\item
$G$ is a semidirect product $S\ltimes N$ where
\item
$N$ is a free nilpotent \Maltsev\ complete group of finite Hirsch length,
\item
$S$ is a subgroup of $\aut(N)$ whose derived subgroup $\gamma_2(S)$ is contained in the group $U$ of $1$-powering automorphims of $N$,
\item $S$ is finitely generated and nilpotent.
\end{itemize}

We need a preliminary result.

\begin{proposition}\label{prop:ref2}
Let $N$ be a \Maltsev\ complete nilpotent group of finite Hirsch length. 
 Let $U$ be the group of $1$-powering automorphisms of $N$. 
 For $\alpha\in\aut(N)$, there is some integer $m_1\ge1$ such that for any multiple $m$ of $m_1$ and any  $m$-powering automorphism $\theta$ of $N$, $\theta\alpha$ and the automorphism of $U$ induced by it  are both integral.
\end{proposition}

\begin{proof}
Let $c$ be the nilpotency class of $N$. For each $2\le j\le c$ let 
$$\beta_{j,\alpha}:\gamma_j(N)/\gamma_{j+1}(N)\to\gamma_j(N)/\gamma_{j+1}(N)$$
be the automorphism induced by $\alpha$ and for $B_j=\hom(N_{\textrm{ab}},\gamma_{j}(N)/\gamma_{j+1}(N)))$ let $\varphi_{j,\alpha}$ be the automorphism of $B_j$ induced by the diagonal action of $\alpha$, i.e.
$$\begin{aligned}
\varphi_{j,\alpha}:B_j&\to B_j\\
\xi&\mapsto \beta_{j,\alpha}\xi\alpha^{-1}.\\
\end{aligned}$$

As all of $N_{\textrm{ab}}$ and $B_j$, $2\le j\le c$ are $\Q$-vector spaces of finite dimension, all the automorphisms $\varphi_{j,\alpha}$ are algebraic and so is the automorphism of $N_{\textrm{ab}}$ induced by $\alpha$ which we also denote $\alpha$ to avoid complications. By Lemma \ref{lem:ref3} taking the least common multiple of the integers involved we may find an integer $m_1$ such that for any multiple $m$ of $m_1$, all the automorphisms $m\alpha$, $m\varphi_{j,\alpha}$ for $2\le j\le c$ are integral. Now, let $\theta$ be an  $m$-powering automorphism  of $N$ and $\alpha_1=\theta\alpha$.  The map of the $j$-th fold tensor product of $N_{\textrm{ab}}$ by itself induced by $\theta$ is just multiplication by $m^j$ and as this tensor product maps onto $\gamma_j(N)/\gamma_{j+1}(N)$ the same holds true for $\gamma_j(N)/\gamma_{j+1}(N)$. 
This implies that $\beta_{j,\alpha_1}=m^j\beta_{j,\alpha}.$ Hence $\varphi_{j,\alpha_1}=m^{j-1}\varphi_{j,\alpha},$ which means that  $\varphi_{j,\alpha_1}$ is integral for $2\le j\le c$. 

Then Lemma \ref{lem:ref1} and Proposition \ref{prop:ref1} imply that the automorphism of $U_{j-1}/U_{j}\leq B_j$ induced by $\alpha_1=\theta\alpha$ is also integral. 
Let $\overline U_j$ be the image of $U_j$ in $U_{\textrm{ab}}$. Then using the fact that $U_{c}=1$ because $N$ is nilpotent of class $c$ (see Proposition \ref{prop:ref1}) we have a series
$$U_{\textrm{ab}}=\overline U_1\ge \overline U_2\ge\dots \ge \overline U_{c}=0$$
such that the automorphism of each factor induced by conjugating by $\alpha_1$ is integral. Hence the second part of the conclusion of the proposition follows from Lemma \ref{lem:ref1}.
\end{proof}

\begin{proof}[Proof that Theorem A holds in this special case]\ 
 
$S$ is finitely generated. Let $\{\alpha_1,\dots,\alpha_r\}$ be a set of generators. Taking the least common multiple of the integers of Proposition \ref{prop:ref2} for each of the $\alpha_i$'s we get an integer $m\ge1$ such that for each $m$-powering automorphism $\varphi$ of $N$, $\varphi\alpha_i$ and the associated automorphism of $U$ are both integral, for $1\le i\le r$.
 Let $\varphi$ be an $m$-powering automorphism of $N$ and consider the group $W:=\langle U\cup\{\varphi\}\rangle$. 
Set $Q:=SW$ and $P:=SU$.
Observe that $U$ is normal in $W$ and $W/U$ is cyclic and central in $Q/U$. Therefore the group $P$ is normal in $Q$.
We now apply Theorem B with these groups $Q$ and $P$ and with $M:=U$, and $L:=S$. The output is a finitely generated subgroup, which we denote by $Q_1$, of $Q$ such that $Q=Q_1U$, $S\le Q_1$, and $Q_1\cap P$ is nilpotent. Now all subgroups of a nilpotent group are subnormal and so $S$ is subnormal in $Q_1\cap P=Q_1\cap SU$. Therefore $S$ is subnormal in $Q_1$ and it follows that $S\ltimes N$ is subnormal in $Q_1\ltimes N$. By Corollary \ref{subnormals}, we have reduced to showing that $$\hd_k(Q_1\ltimes N)=h(Q_1\ltimes N).$$ The result follows from Proposition \ref{fpinftys} because, as we shall now see, the group $Q_1\ltimes N$ is locally of type $\fpinfty$. 

The group $Q_1$ contains an $m$-powering automorphism $\theta$ (equal to $\varphi$ module $U$). Moreover, $Q_1$ is generated by the set $\{\alpha_1,\dots,\alpha_r,\theta\}\cup Y$ where $Y$ is a finite subset of $U$. So it is also generated by the set 
$X:=\{\theta\alpha_1,\dots,\theta\alpha_r,\theta\}\cup Y$. By Proposition \ref{prop:ref2} both $\theta\alpha_i$ and the automorphism of $U$ that it induces are integral for each $i$.
 Notice that the group $M:=(Q_1\cap U)\ltimes N$ is nilpotent and $Q_1\ltimes N$ is generated by $M\cup X$. Also by Lemma \ref{lem:nuevo4} the automorphisms of $Q_1\cap U$ induced by conjugating by the elements of $X$ are integral, which implies that the same is true of the automorphisms induced on $M$. Furthermore, since $S'\le U$ we have $Q'\le U$ which means that $Q_1/(Q_1\cap U)$ is abelian. Therefore we may apply Theorem C to this setup to deduce that $Q_1\ltimes N$ is locally of type $\fpinfty$.
\end{proof}

\section{Proof of Theorem A}

\smallskip

Let $G$ be a group with $\hd_k(G)<\infty$. By Lemma \ref{bk},  $G$ has finite Hirsch length. Therefore by Lemma \ref{hl} $G$ has a locally finite normal subgroup $T$ such that $G/T$ is torsion free nilpotent-by-free abelian of finite rank-by finite. 
We  know that
\begin{itemize}
\item
$\hd_k(G)=\hd_k(G/T)$ for any locally finite normal subgroup $T$ of $G$, and
\item
$\hd_k(G)=\hd_k(H)$ for any subgroup $H$ of finite index in $G$,
\end{itemize}
(see Proposition \ref{tipsy}).
Obviously, the same happens for the Hirsch length and therefore we can replace $G$ by a section which has a torsion-free nilpotent normal subgroup $E$  such that $A:=G/E$ is free abelian of finite rank. We can embed $G$ into a larger group $\widehat G$ fitting into the  commutative diagram
\[
\xymatrix{
E\ar@{>->}[r]\ar[d]&G\ar@{>>}[r]\ar[d]&A\ar@{=}[d]\\
E^\Q\ar@{>->}[r]&\widehat G\ar@{>>}[r]&A.
}
\]
For the details of this construction see \cite[Proposition 1.1]{Hilton}. 
We then have $\hd_K(G)=\hd_K(\widehat G)$ and  $h(G)=h(\widehat G)$ thus we can replace $G$ by $\widehat G$ and assume that $E$ is already \Maltsev\ complete.

Using Theorem B with $Q:=P:=G$, $M:=E$ and $L:=1$ we may find a finitely generated nilpotent (thus polycyclic) subgroup $H$ of $G$ so that $G=HE$. Consider now the semidirect product $H\ltimes E$. There is a surjective map $H\ltimes E\twoheadrightarrow G=HE$ sending $(h,e)$ to $he$ so an application of  Lemma \ref{LHS} implies that we can reduce the problem to the group $H\ltimes E$. Now, let $C_H(E)$ be the kernel in $H$ of the conjugacy action on $E$. This subgroup is normal in $H\ltimes E$ and the quotient map is $H\ltimes E\twoheadrightarrow H/C_H(E)\ltimes E$. Put $X=H/C_H(E)$. As the group $H$ is polycyclic, so is $C_H(E)$ thus
we have that 
$$\hd_K(H\ltimes E)=\hd_K(X\ltimes E)+\hd_K(C_H(E))$$
which together with the fact that $\hd_K(C_H(E))=h(C_H(E))$ because it is polycyclic implies that we can further reduce the problem to the group $X\ltimes E$. As $X$ acts faithfully on $E$, we may see it as a subgroup of the group $\aut(E)$. 
From this point of view, $[X,X]$ consists exclusively of $1$-powering automorphisms of $E$.

Using Proposition \ref{Maltsevcover}, we choose a free \Maltsev\ complete group 
$\widehat E$ with the same class so that there is a surjective homomorphism $\pi:\widehat E\to E$ that induces an isomorphism $\widehat E/\gamma_2(\widehat E)\cong E/\gamma_2(E)$. There is an induced epimorphism 
$$\bar\pi:\aut(\widehat E;\ker\pi)\twoheadrightarrow\aut(E)$$
whose kernel $Z$ is contained in the subgroup $U(\aut(\widehat E))$ of $1$-powering automorphisms. Note that $Z$ is nilpotent and by Lemma \ref{new3.12} \Maltsev\ complete. 

Let $T$ denote the preimage of $X$ under the map $\bar\pi$. 
From Proposition 3.6(ii) we see that $[T,T]$ consists entirely of $1$-powering automorphisms of $\widehat E$.
At this point we may apply again Theorem B with $T$ playing the role of both $P$ and $Q$, with $M:=Z$ and with $L:=\{1\}$ and we deduce that there is some $S\leq T$ finitely generated and nilpotent such that $SZ=T$. 
We can now consider the  semidirect product $S\ltimes \widehat E$ which collapses naturally via $\pi$ and $\bar\pi$ onto $X\ltimes E$. By Lemma \ref{LHS} it suffices to prove that 
$$\hd_k(S\ltimes \widehat E)=h(S\ltimes \widehat E).$$
We have reduced to the special case that was considered in preceding section and Theorem A follows.

\bibliographystyle{abbrv}

\end{document}